\documentclass[11pt]{article}
\usepackage{amsmath}
\usepackage{amssymb}
\usepackage{latexsym,bm}
\usepackage{amsthm}

\usepackage[top=2.4cm,bottom=2cm,left=2.4cm,right=2.4cm]{geometry}

\usepackage{cite}

\providecommand{\U}[1]{\protect\rule{.1in}{.1in}}
\providecommand{\U}[1]{\protect\rule{.1in}{.1in}}
\providecommand{\U}[1]{\protect\rule{.1in}{.1in}}
\providecommand{\U}[1]{\protect\rule{.1in}{.1in}}

\newtheorem{theorem}{Theorem}[section]
\newtheorem{corollary}[theorem]{Corollary}
\newtheorem{lemma}{Lemma}[section]

\theoremstyle{definition}
\newtheorem{definition}{Definition}[section]
\newtheorem{remark}{Remark}[section]
\numberwithin{equation}{section}
\theoremstyle{example}

\numberwithin{equation}{section}

\begin{document}

\title{Convergence analysis of an inexact inertial Krasnoselskii-Mann algorithm with applications}

\author{ Fuying Cui$^{a}$, Yang Yang$^{a}$,  Yuchao Tang$^{a}$\footnote{Corresponding author: Yuchao Tang, Email address: hhaaoo1331@163.com}, Chuanxi Zhu$^{a}$
\\
{\small ${^a}$ Department of Mathematics, Nanchang University,  Nanchang 330031, P.R. China }}

 \date{}

\maketitle

{}\textbf{Abstract.}\
The classical Krasnoselskii-Mann iteration is broadly used for approximating fixed points of nonexpansive operators. To accelerate the convergence of the Krasnoselskii-Mann iteration, the inertial methods were received much attention in recent years. In this paper, we propose an inexact inertial Krasnoselskii-Mann algorithm. In comparison with the original inertial Krasnoselskii-Mann algorithm, our algorithm allows error for updating the iterative sequence, which makes it more flexible and useful in practice. We establish weak convergence results for the proposed algorithm under different conditions on parameters and error terms. Furthermore, we provide a nonasymptotic convergence rate for the proposed algorithm. As applications, we propose and study inexact inertial proximal point algorithm and inexact inertial forward-backward splitting algorithm for solving monotone inclusion problems and the corresponding convex minimization problems.

\textbf{Key words}: Nonexpansive operator; Krasnoselskii-Mann iteration; inertial Krasnoselskii-Mann iteration; inexact Krasnoselskii-Mann iteration.

\textbf{AMS Subject Classification}:  47H09; 47H05; 90C25.

\section{Introduction}
\label{sec-introduction}

Throughout the paper, let $H$ be a real Hilbert space, which equipped with inner product  $\langle\cdot,\cdot\rangle$ and norm $\|\cdot\|$.
 We denote by $Fix(T)$ that the fixed points set of an operator $T$, more precisely,
$Fix(T)=\{x\in H : Tx=x\}$.

Many problems in pure and applied mathematics can be formulated as fixed point problems. The fixed point problem of nonexpansive operators was received much attention in recent years. See for example \cite{Chang2002book,xuhk2006,berinde2007,chidumebook2009,cegielskibook,Tang2013MC,Tang2014MMAS} and references therein. Many efficient iterative algorithms for solving composite convex optimization problems include the primal-dual fixed point proximity algorithm \cite{chenpj2013,TangJCM2019}, the Davis-Yin's three-operator splitting algorithm \cite{davis2015,Zong2018} and the primal-dual hybrid gradient algorithm and its variants \cite{chambolle2016AN,tang2016-1} that can be formulated as a fixed point problem of nonexpansive operators.

The classical methods for solving the fixed point of nonexpansive operators is the Krasnoselskii-Mann (KM) iteration scheme, which is defined by,
\begin{equation}\label{KM-iteration}
x^{k+1}= (1-\lambda_{k})x^{k}+\lambda_{k}Tx^{k},
\end{equation}
where $T$ is a nonexpansive operator. The convergence of the KM iteration (\ref{KM-iteration}) is well studied in Hilbert spaces. In fact, the sequence $\{x^{k}\}$ generated by the KM iteration (\ref{KM-iteration}) converges weakly to a fixed point of $T$, under the condition $\{\lambda_{k}\}\subset [0,1]$ and $\sum_{k=0}^{+\infty}\lambda_{k}(1-\lambda_{k})=+\infty$. We refer interested readers to Theorem 5.15 of \cite{bauschkebook2017} for detail of proof. Recently, several authors provided the convergence rate analysis for the KM algorithm (\ref{KM-iteration}) in the sense of the difference between $x^k$ and $Tx^k$. See, for example \cite{Cominetti2014,Matsushita2017}.

In recent years, inertial methods have become more and more popular. Various inertial  algorithms were studied, see for example \cite{Alvarez2001,Moudafi2003,Dong2016Optim,Attouch2018,Cui2019} and references therein. The inertial method is also called the heavy ball method, which is based on a discretization of a second-order dissipative dynamic system. Maing\'{e}, in \cite{Mainge2008JCAM}, proposed the following inertial Krasnoselskii-Mann (iKM) algorithm,
\begin{equation}\label{inertial KM}
\left\{
\begin{aligned}
\omega^{k} &=x^{k}+\alpha_{k}(x^{k}-x^{k-1}) \\
x^{k+1} &=(1-\lambda_{k})\omega^{k}+\lambda_{k}T\omega^{k}.
\end{aligned}
\right.
\end{equation}
The convergence of (\ref{inertial KM}) is proved under the condition that:
\begin{equation}\label{inertial-parameter-condition}
\begin{aligned}
& \alpha_k \in [0, \alpha), \textrm{ where } \alpha\in [0,1), \textrm{ and }
& \sum_{k=0}^{+\infty} \alpha_k \| x^k - x^{k-1} \|^2 < +\infty,
\end{aligned}
\end{equation}
and
\begin{equation}
0 < \inf \lambda_k \leq \sup \lambda_k < 1.
\end{equation}
The difference between the KM iteration  (\ref{KM-iteration}) and the iKM iteration (\ref{inertial KM}) is that the latter used a combination of the iterative sequences $x^{k}$ and  $x^{k-1}$ to obtain the new iterative sequence. Many numerical experiment results confirm that the iKM iteration converges faster than the KM iteration without inertial.  In \cite{Bot2015AMC}, Bo\c{t} et al. also studied the convergence of the iKM iteration (\ref{inertial KM}). But they removed   the second condition in (\ref{inertial-parameter-condition}). As a supplement, they require a strict condition on the parameters of $\alpha_k$ and $\lambda_k$, which is given by
\begin{equation}
\delta > \frac{\alpha^{2}(1+\alpha)+\alpha \sigma}{1-\alpha^{2}} \quad \textrm{ and } \quad 0<\lambda \leq \lambda_{k}\leq \frac{\delta-\alpha[\alpha (1+\alpha)+\alpha \delta+\sigma]}{\delta[1+\alpha (1+\alpha)+\alpha \delta+\sigma]},
\end{equation}
where $\lambda, \sigma > 0 $ and $\{\alpha_k\}$ is nondecreasing with $0\leq \alpha_{k} \leq \alpha <1$.
As an application, they proposed an inertial Douglas-Rachford (iDR) algorithm. Further, an inertial alternating direction method of multipliers (ADMM) was developed in \cite{Bot2016MTA}. Some recent generalization of the iKM (\ref{inertial KM}) can be found in \cite{JohnstoneCOA2017, Dong2018Book, DongJGO2019,WuzmCOA2019}.

On the other hand, to incorporate numerical errors in the computation of the iterative sequence, the KM iteration (\ref{KM-iteration}) was generalized to the inexact case. More precisely, the inexact KM iteration is defined by
\begin{equation}\label{inexact-KM}
x^{k+1}=(1-\lambda_{k})x^{k}+\lambda_{k}(Tx^{k}+e_{k}),
\end{equation}
where $\lambda_{k}\in (0,1)$ and $e_{k}$ is an error term. It is obvious that if $e_k =0$ in (\ref{inexact-KM}), then it reduces to the classical KM iteration (\ref{KM-iteration}). The inexact KM iteration (\ref{inexact-KM}) has wide application in the study of operator splitting algorithms.
 The convergence of the inexact proximal point algorithm studied in \cite{RockafellarSIAM1976} could be easily obtained from the inexact KM iteration (\ref{inexact-KM}). Besides, the inexact forward-backward splitting algorithm \cite{combettes2005}, the inexact Douglas-Rachford algorithm \cite{combettes2007} and the inexact three-operator splitting algorithm \cite{Zong2018} could also be derived from the convergence analysis of the inexact KM iteration scheme (\ref{inexact-KM}. Besides, Combettes \cite{Combettes2004Optimization} investigated the convergence of the inexact KM (\ref{inexact-KM}) involves nonstationary compositions of perturbed averaged operators. As a direct application, a nonstationary forward-backward splitting algorithm with errors for solving monotone inclusion of the sum of two maximally monotone operators with one of them is inverse strongly monotone was obtained. See also \cite{Combettes2015JMAA,Combettes2017,CuiJIA2019}. Following the convergence rate analysis of the KM algorithm (\ref{KM-iteration}), Liang et al. \cite{Liang2016F} studied the convergence rate of the inexact KM algorithm (\ref{inexact-KM}).
 It is worth mentioning that the iterative sequence generated by the KM iteration (\ref{KM-iteration}) or the inexact KM iteration (\ref{inexact-KM}) is Fej\'{e}r-monotone or quasi-Fej\'{e}r-monotone to the fixed points set of $T$, while the iKM iteration (\ref{inertial KM}) doesn't have this property.

The purpose of this paper is to combine the inertial method with the inexact method. We aim to propose an inexact inertial Krasnoselskii-Mann algorithm (see (\ref{two-step algorithm1})). Further, we investigate the asymptotic behavior of the proposed algorithm for finding fixed points of nonexpansive operators under different conditions on parameters and error terms. Since the iterative sequence generated by the inertial algorithm doesn't have the Fej\'{e}r-monotone property. To overcome this difficulty,  we develop a new technique to prove the weak convergence of the proposed algorithm. We also study the convergence rate of the proposed algorithm in the spirit of the recent work of Shehu \cite{Shehu2018NFAO}. As applications, we obtain an inexact inertial proximal point algorithm and an inexact inertial forward-backward splitting algorithm for solving monotone inclusion problems and convex optimization problems. These iteration algorithms are completely new and haven't appeared in the literature before.

The rest of this paper is organized as follows. Section 2, we recall some definitions and lemmas on nonexpansive operators and monotone operator theory. Section 3, we propose an inexact inertial Krasnoselskii-Mann algorithm and analyze its convergence property. Section 4, we present several applications on monotone inclusion problems with the proposed algorithm. Finally, we give some conclusions and future works.

\section{Preliminaries}
\label{sec:pre}

In this section, we recall some  definitions  and  preliminary results on nonexpansive operators and monotone operators theory in Hilbert space. Let $H$ be a real Hilbert space with inner product $\langle \cdot , \cdot \rangle$ and norm $\|\cdot\|$. We define $x^k \rightharpoonup x$ denotes the sequence $\{x^k\}$ converges weakly to $x$ and $x^k \rightarrow x$ to indicate that the sequence $\{x^k\}$ converges strongly to $x$. Further, $\omega_{w}(x^k)$ denotes every sequential weak cluster point of $\{x^k\}$.

\begin{definition}(\cite{bauschkebook2017})
Let $C$ be a nonempty subset of $H$. Let $T:C\rightarrow H$, then

\noindent (i) $T$ is called nonexpansive, if
$$
\|T\mu-T\nu\| \leq \|\mu-\nu\|, \quad \forall \mu,\nu\in C.
$$

\noindent (ii) $T$ is called firmly nonexpansive, if
$$
\|T\mu-T\nu\|^2 \leq \|\mu-\nu\|^2 - \| (I-T)\mu- (I-T)\nu \|^2, \quad \forall \mu,\nu\in C,
$$
or equivalently
$$
\|T\mu-T\nu\|^2 \leq \langle T\mu - T\nu, \mu - \nu \rangle, \quad \forall \mu,\nu\in C.
$$

\noindent (iii) $T$ is called $\theta$-averaged, $\theta\in (0,1)$, if there exists an nonexpansive mapping $S$ such that $T = (1-\theta)I + \theta S$.
\end{definition}

It is easy to prove that every averaged operator and firmly nonexpansive operator are nonexpansive operators. Also, a firmly nonexpansive operator is $1/2$-averaged.

\begin{definition}(\cite{bauschkebook2017})
Let $A:H\rightarrow 2^H$ be a set-valued operator. $A$ is said to be monotone, if
$$
\langle u-v, x-y \rangle \geq 0, \quad \forall x,y\in H, u\in Ax, v\in Ay.
$$
Further, $A$ is said to be maximally monotone, if its graph is not strictly contained in the graph of any other monotone operator on $H$.
\end{definition}

\begin{definition}(\cite{Byrne2004})
Let $B:H\rightarrow H$ be a single-valued operator.  Then $B$ is called $\beta$-inverse strongly monotone, for some $\beta \in (0, +\infty)$, if
$$
\langle Bx-By,x-y \rangle \geq \beta \|Bx-By\|^{2}, \quad \forall x,y\in H.
$$
\end{definition}

\begin{definition}(\cite{bauschkebook2017})
Let $A:H\rightarrow 2^H$ be a maximally monotone operator. The resolvent operator of $A$ with index $\gamma >0$ is defined as
$$
J_{\gamma A} = (I+\gamma A)^{-1}.
$$
where $I$ is identity operator on $H$.
\end{definition}

It follows from Proposition 23.8 of \cite{bauschkebook2017} that the resolvent operator $J_{\gamma A}$ of a maximally monotone operator $A$ with index $\gamma >0$ is firmly nonexpansive.

Let $f:H\rightarrow (-\infty,+\infty]$ is a proper lower semi-continuous convex function. The subdifferential of $f$ is the set $\partial f(x) = \{u\in H | f(y) \geq f(x) + \langle u, y-x \rangle, \forall y\in H\}$. Let $A = \partial f$, then $J_{\gamma \partial f} = prox_{\gamma f}$. Here $ prox_{\gamma f}$ denotes the proximity operator of $f$ with index $\gamma >0$, which is defined by
$$
prox_{\gamma f}(u) = \arg\min_{x} \left\{  \frac{1}{2\gamma}\|x-u\|^2 + f(x)  \right\}.
$$

\begin{definition}(\cite{Combettes2004Optimization})
Let $C$ be a nonempty subset of $H$ and let $\{x^k\}$ be a sequence in $H$. Then

\noindent (i) $\{x^k\}$ is Fej\'{e}r-monotone with respect to $C$, if
$$
\|x^{k+1}-x\| \leq \|x^k -x\|, \quad \forall x\in C, \forall k\geq 0.
$$

\noindent (ii) $\{x^k\}$ is quasi-Fej\'{e}r-monotone with respect to $C$, if
$$
\|x^{k+1}-x\| \leq \|x^k -x\| + \varepsilon_k, \quad \forall x\in C, \forall k\geq 0,
$$
where $\sum_{k=0}^{+\infty}\varepsilon_k < +\infty$.
\end{definition}

\begin{lemma}(Demiclosedness principle)\label{our-lemma1}(\cite{bauschkebook2017})
Let $H$ a be Hilbert space. Let $C$ be a nonempty closed convex set of $H$ and $\{x^{k}\}_{n\in N}$ be sequence in $C$. Let
$T:D\rightarrow H$ be nonexpansive. Let $x\in H$ such that $x^{k}\rightharpoonup x$ and $T(x^{k})-x^{k} \rightarrow 0$ as $k\rightarrow +\infty$.
Then
$x\in Fix(T)$.
\end{lemma}

The following lemmas play an important role in the proof of the convergence of the proposed  algorithm.

\begin{lemma}\label{our-lemma2}(\cite{Alvarez2001})
Let $\{\psi^{k}\}$, $\{\delta_{k}\}$ and $\{\alpha_{k}\}$ be sequences in $[0,+\infty)$ such that

\noindent \emph{(a)} $\psi^{k+1}-\psi^{k}\leq \alpha_{k}(\psi^{k}-\psi^{k-1})+\delta_{k}$, $\forall k\geq 1$;

\noindent \emph{(b)} $\sum_{k=1}^{+\infty}\delta_{k} < +\infty$;

\noindent \emph{(c)} there exists a real number $\alpha \in [0,1)$ with $\alpha_{k} \subset[0,\alpha)$.

Then $\psi^{k}$  is convergent. Moreover $\sum_{k=1}^{+\infty}[\psi^{k+1}-\psi^{k}]_{+} < +\infty$,
where $[m]_{+}= \max \{m,0\}$.
\end{lemma}

\begin{lemma}\label{our-lemma3}(\cite{bauschkebook2017})
Let $C$ be a nonempty subset of $H$ and $\{x^k\}$ be a sequence in $H$ such that the following conditions:

\noindent (i) for every $x\in C$, $\lim_{k\rightarrow +\infty}\|x^k - x\|$ exists;

\noindent (ii) $\omega_{w}(x^k)\subseteq C$.

Then $\{x^k\}$ converges weakly to a point in $C$.

\end{lemma}

\section{An inexact  inertial Krasnoselskii-Mann algorithm}

In this section, we propose an  inexact inertial Krasnoselskii-Mann algorithm for computing fixed points of nonexpansive operators. We study the convergence and the convergence rate of the proposed algorithm under two different conditions. Now, we are ready to present our  main algorithm:

\begin{equation}\label{two-step algorithm1}
\left\{
\begin{aligned}
\mu^{k}&=z^{k}+\alpha_{k}(z^{k}-z^{k-1}), \\
z^{k+1}&=\mu^{k}+\lambda _{k}(T\mu^{k}+e^{k}-\mu^{k}).
\end{aligned}
\right.
\end{equation}

\begin{theorem}\label{main-theorem 1}
Let $H$ be a real Hilbert space. Let $T:H\rightarrow H$ is a nonexpansive operator such that $Fix(T)\neq \varnothing$. For any given $z^{0}, z^{-1}\in H$, let
the iterative sequences $\{z^{k}\}$ and $\{\mu^{k}\}$ are generated by the iteration scheme (\ref{two-step algorithm1}).
Assume that the parameters $\lambda_{k}, \alpha_{k}$ satisfy the conditions (I) of:

\noindent \emph{(a)}\,  $ 0\leq \alpha_{k} \leq \alpha < 1 $, $ 0\leq \lambda \leq \lambda_{k}\leq\lambda^{'} < 1 $;

\noindent \emph{(b)}\, for every $k\geq 0$, $\sum _{k=0}^{+\infty}\alpha_{k}\|z^{k+1}-z^{k}\|^{2} < +\infty$, and $\sum _{k=0}^{+\infty}\lambda_{k}\|e^{k}\| < +\infty$;

\noindent \emph{(c)}\, $\{z^{k}\}$ is bounded.

Then the following hold:

\noindent \emph{(i)} $\lim_{k\rightarrow +\infty}\|z^{k}-z^{*}\|$ exists, for any $z^{*}\in Fix(T)$.

\noindent \emph{(ii)} $\lim_{k\rightarrow +\infty}\|T\mu^{k}-\mu^{k}\|=0$.

\noindent \emph{(iii)} $\{z^{k}\}$ converges weakly to a fixed point of $T$.
\end{theorem}

\begin{proof}
(i) For the sake of convenience, we define
\begin{equation}\label{eq3-1}
w^{k} =\mu^{k}+\lambda_{k}(T\mu^{k}-\mu^{k}).
\end{equation}
Then, the iterative sequence $\{z^{k+1}\}$ in (\ref{two-step algorithm1}) can be rewritten as
\begin{equation}\label{eq3-2}
z^{k+1} = w^{k}+\lambda_{k}e^{k}.
\end{equation}

Let $z^{*}\in Fix(T)$. By (\ref{eq3-2}), we have
\begin{align}\label{eq3-4}
\|z^{k+1}-z^{*}\|^{2} &=\|w^{k}-z^{*}+\lambda_{k}e^{k}\|^{2} \nonumber\\
& \leq \|w^{k}-z^{*}\|^{2}+ 2\langle z^{k+1}-z^{*},\lambda_{k}e^{k}\rangle  \nonumber\\
&\leq \|w^{k}-z^{*}\|^{2}+2\lambda_{k}\|z^{k+1}-z^{*}\|\|e^{k}\|,
\end{align}
where the first inequality comes from the fact that $\|x+y\|^2 \leq \|x\|^2 + 2 \langle x+y, y\rangle$, for any $x,y\in H$,
 and the second inequality is due to the Cauchy-Schwartz inequality. With the help of the equality
 $\|(1-\alpha)x+\alpha y\|^{2}=(1-\alpha)\|x\|^{2}+\alpha \|y\|^{2}-\alpha (1-\alpha)\|x-y\|^{2}$, for any $\alpha \in R$, and $x,y\in H$. Then, by (\ref{eq3-1}), we have
\begin{align}\label{eq3-5}
\|\omega^{k}-z^{*}\|^{2} & = (1-\lambda_{k})\|\mu^{k}-z^{*}\|^{2}+\lambda_{k}\|T\mu^{k}-z^{*}\|^{2}-\lambda_{k}(1-\lambda_{k})\|T\mu^{k}-\mu^{k}\|^{2} \nonumber\\
&\leq  (1-\lambda_{k})\|\mu^{k}-z^{*}\|^{2}+\lambda_{k}\|\mu^{k}-z^{*}\|^{2}-\lambda_{k}(1-\lambda_{k})\|T\mu^{k}-\mu^{k}\|^{2}\nonumber\\
& = \|\mu^{k}-z^{*}\|^{2} -\lambda_{k}(1-\lambda_{k})\|T\mu^{k}-\mu^{k}\|^{2}\nonumber\\
&=(1+\alpha_{k})\|z^{k}-z^{*}\|^{2}-\alpha_{k}\|z^{k-1}-z^{*}\|^{2}\nonumber\\
&+\alpha_{k}(1+\alpha_{k})\|z^{k}-z^{k-1}\|^{2}-\lambda_{k}(1-\lambda_{k})\|T\mu^{k}-\mu^{k}\|^{2}.
\end{align}

Let $\delta_{k}=2\|z^{k+1}-z^{*}\|$, and $\rho_{k}=\alpha_{k}(1+\alpha_{k})$.
Then, we arrive
\begin{align}\label{eq3-6}
\|z^{k+1}-z^{*}\|^{2}&\leq(1+\alpha_{k})\|z^{k}-z^{*}\|^{2}-\alpha_{k}\|z^{k-1}-z^{*}\|^{2}+\delta_{k}\lambda_{k}\|e^{k}\| \nonumber\\
&+\rho_{k}\|z^{k}-z^{k-1}\|^{2}-\lambda_{k}(1-\lambda_{k})\|T\mu^{k}-\mu^{k}\|^{2},
\end{align}
which implies that
\begin{align}\label{eq3-7}
\|z^{k+1}-z^{*}\|^{2}-\|z^{k}-z^{*}\|^{2}&\leq\alpha_{k}(\|z^{k}-z^{*}\|^{2}-\|z^{k-1}-z^{*}\|^{2})+\delta_{k}\lambda_{k}\|e^{k}\| \nonumber\\
&+\rho_{k}\|z^{k}-z^{k-1}\|^{2}.
\end{align}

Notice that the condition (a), (b) and (c), we have $\sum _{k=0}^{\infty}\rho_{k}\|z^{k}-z^{k-1}\|^{2} < +\infty$
and $\{\delta_{k}\}$ is bounded. By Lemma \ref{our-lemma2}, we have $\lim_{k\rightarrow +\infty}\|z^{k}-z^{*}\|$ exists and
$\sum _{k=0}^{\infty}[\|z^{k}-z^{*}\|^{2}- \|z^{k-1}-z^{*}\|^{2}]_{+}< +\infty$.

(ii) It follows from (\ref{eq3-6}) that
\begin{align}\label{eq3-8}
 \lambda_{k}(1-\lambda_{k})\|T\mu^{k}-\mu^{k}\|^{2}&\leq \|z^{k}-z^{*}\|^{2}-\|z^{k+1}-z^{*}\|^{2}+\alpha_{k}(\|z^{k}-z^{*}\|^{2}-\|z^{k-1}-z^{*}\|^{2})\nonumber\\
 &+\delta_{k}\lambda_{k}\|e^{k}\|+\rho_{k}\|z^{k}-z^{k-1}\|^{2}.
\end{align}
Letting $k\rightarrow +\infty$ in the above inequality, we obtain that $\lim_{k\rightarrow +\infty}\|T\mu^{k}-\mu^{k}\|=0$.

(iii)  From (\ref{eq3-1}),  we have $\|\omega^{k}-\mu^{k}\|=\lambda_{k}\|T\mu^{k}-\mu^{k}\| \leq \lambda^{'} \|T\mu^{k}-\mu^{k}\|$.  By (ii), we get $\lim_{k\rightarrow +\infty}\|\omega^{k}-\mu^{k}\|=0$. Then
\begin{align}\label{eq3-9}
 \|z^{k+1}-\mu^{k}\|&=\|z^{k+1}-\omega^{k}+\omega^{k}-\mu^{k}\| \nonumber\\
 &\leq \|z^{k+1}-\omega^{k}\|+ \|\omega^{k}-\mu^{k}\| \nonumber\\
 &\leq \lambda_k \|e_{k}\|+\|\omega^{k}-\mu^{k}\|.
\end{align}
According to condition (b) and $\lim_{k\rightarrow +\infty}\|\omega^{k}-\mu^{k}\|=0$, we get
\begin{equation}\label{eq3-10}
\lim_{k\rightarrow +\infty}\|z^{k+1}-\mu^{k}\|=0.
\end{equation}
Further, we have
\begin{align}\label{eq3-14}
\|z^{k+1}-z^{k}\|&\leq \|z^{k+1}-\mu^{k}\|+\|\mu^{k}-z^{k}\| \nonumber\\
&\leq \|z^{k+1}-\mu^{k}\|+\alpha_{k}\|z^{k}-z^{k-1}\|\nonumber\\
&\rightarrow 0    \textrm{ as } k \rightarrow +\infty,
\end{align}
and $\|\mu^{k}-z^{k}\|\leq \|\mu^{k}-z^{k+1}\|+\|z^{k+1}-z^{k}\| \rightarrow 0    \textrm{ as } k \rightarrow +\infty$.

Next, we prove that $\omega_{w}(z_{k})\subseteq Fix(T)$. In fact, let $\bar{z}\in \omega_{w}(z_{k})$, such that $z^{k_{n}}\rightharpoonup \bar{z}$. Since
$\lim_{k\rightarrow +\infty}\|\mu^{k}-z^{k}\|=0$, then $\mu^{k_{n}}\rightharpoonup \bar{z}$. Notice that $\lim_{k\rightarrow +\infty}\|T\mu^{k}-\mu^{k}\|=0$, it follows from the demiclosedness of nonexpansive operators, we have $\bar{z}\in Fix (T)$. That is $\omega_{w}(z_{k})\subseteq Fix(T)$. By Lemma \ref{our-lemma3}, we can conclude that $\{z_{k}\}$ converges weakly to a fixed point of $T$. This completes the proof.

\end{proof}

In the following, we prove the convergence of the proposed iteration scheme (\ref{two-step algorithm1}) by removing the condition $\sum_{k=0}^{+\infty}\alpha_k \|z^{k+1}- z^k\|^2<+\infty$ in Theorem \ref{main-theorem 1}. However, we have to add more strict conditions on the relaxation parameters $\lambda_k$ and the error term $e^k$.

\begin{theorem}\label{main-theorem 2}
Let $H$ be a real Hilbert space. Let $T:H\rightarrow H$ is a nonexpansive operator such that $Fix(T)\neq \varnothing$. For any given $z^{0}, z^{-1}\in H$,
let the iterative sequences $\{z^{k}\}, \{\mu^{k}\}$ are defined by (\ref{two-step algorithm1}). Assume that the parameters $\{\lambda_{k}\}, \{\alpha_{k}\}$ are nondecreasing and satisfy the conditions (II) of:

\noindent \emph{(a)}  $ 0\leq  \alpha_{k} \leq \alpha < 1 $ with $\alpha_0 =0$ , and $\alpha_{k}$ is nondecreasing;

\noindent \emph{(b)}  Let $\lambda, \sigma, \delta > 0 $ such that
\begin{equation}
\delta > \frac{\alpha[\alpha (1+\alpha)+\sigma]}{1-\alpha^{2}} \textrm{ and } 0<\lambda \leq \lambda_{k}\leq \frac{\delta-\alpha[\alpha (1+\alpha)+\alpha \delta+\sigma]}{\delta[1+\alpha (1+\alpha)+\alpha \delta+\sigma]};
\end{equation}

\noindent \emph{(c)} $\{z^{k}\}$ is bounded;

\noindent \emph{(d)} $\sum_{k=0}^{+\infty}\|e^k\| < +\infty$.

Then the following hold :

\noindent \emph{(i)} $\sum _{k=0}^{\infty}\|z^{k+1}-z^{k}\|^{2} < +\infty$. Moreover,
$\lim_{k\rightarrow +\infty}\|z^{k+1}-z^{k}\|=0$.

\noindent \emph{(ii)} $\lim_{k\rightarrow +\infty}\|z^{k}-z^{*}\|$ exists, for any $z^{*}\in Fix(T)$.

\noindent \emph{(iii)} $\lim_{k\rightarrow +\infty}\|T\mu^{k}-\mu^{k}\|=0$ and $\{z^{k}\}$ converges weakly to a fixed point of $T$.
\end{theorem}

\begin{proof}
(i) The same to Theorem \ref{main-theorem 1}, we define
\begin{equation}
w^{k} =\mu^{k}+\lambda_{k}(T\mu^{k}-\mu^{k}).
\end{equation}
Then, we obtain
\begin{equation}
z^{k+1} =w^{k}+\lambda_{k}e^{k}.
\end{equation}

According to $z^{k+1}=\mu^{k}+\lambda _{k}(T\mu^{k}+e^{k}-\mu^{k})$, we have
\begin{align}\label{eq32-1}
\|T\mu^{k}-\mu^{k}\|^{2} & = \left \|\frac{z^{k+1}-\mu^{k}}{\lambda_{k}}-e^{k} \right\|^{2} \nonumber\\
& = \left \|  \frac{z^{k+1}-\mu^k}{\lambda_k}    \right\|^2 - 2 \left\langle  \frac{z^{k+1}-\mu^k}{\lambda_k}, e^k    \right\rangle + \|e^k\|^2 \nonumber \\
& \geq \left \|  \frac{z^{k+1}-\mu^k}{\lambda_k}    \right\|^2 -  2 \left\|  \frac{z^{k+1}-\mu^k}{\lambda_k}  \right\|  \|e^k\| + \|e^k\|^2 \nonumber \\
& = \left \|\frac{z^{k+1}-z^{k}}{\lambda_{k}}+\frac{\alpha_{k}(z^{k-1}-z^{k})}{\lambda_{k}}\right\|^{2} -  \frac{ 2}{\lambda_k}\|  z^{k+1}-\mu^k \|  \|e^k\| + \|e^k\|^2 \nonumber \\
&\geq \frac{\|z^{k+1}-z^{k}\|^{2}}{\lambda_{k}^{2}}+\frac{\alpha_{k}^{2}\|z^{k-1}-z^{k}\|^{2}}{\lambda_{k}^{2}} +\frac{\alpha_{k}}{\lambda_{k}^{2}} (-\tau_{k}\|z^{k+1}-z^{k}\|^{2}-\frac{1}{\tau_{k}} \|z^{k-1}-z^{k}\|^{2})\nonumber\\
&-   \frac{2}{\lambda_k}\|  z^{k+1}-u^k \|  \|e^k\| + \|e^k\|^2,
\end{align}
where $\tau_{k}=\frac{1}{\alpha_{k}+\delta\lambda_{k}}$.

Let $\psi^{k+1}=\|z^{k+1}-z^{*}\|^{2}$, it follows from (\ref{eq3-6}) and (\ref{eq32-1}), we obtain
\begin{align}\label{eq32-2}
\psi^{k+1}-\psi^{k}-\alpha_{k}(\psi^{k}-\psi^{k-1}) & \leq \delta_{k}\lambda_{k}\|e^{k}\|
+\rho_{k}\|z^{k}-z^{k-1}\|^{2}\nonumber\\
&-\lambda_{k}(1-\lambda_{k})[\frac{\|z^{k+1}-z^{k}\|^{2}}{\lambda_{k}^{2}}+\frac{\alpha_{k}^{2}\|z^{k-1}-z^{k}\|^{2}}{\lambda_{k}^{2}}\nonumber\\ &-\frac{\alpha_{k}\tau_{k}}{\lambda_{k}^{2}}\|z^{k+1}-z^{k}\|^{2}-\frac{\alpha_{k}}{\tau_{k}\lambda_{k}^{2}} \|z^{k-1}-z^{k}\|^{2}+\|e^{k}\|^{2}\nonumber\\
&-   \frac{2}{\lambda_k}\|  z^{k+1}-u^k \|  \|e^k\| ]\nonumber\\
& \leq M_k \|e^{k}\| + \varsigma_{k} \| z^k - z^{k-1} \|^2 + \omega_{k} \| z^{k+1} -z^k \|^2,
\end{align}
where $\varsigma_{k}=\rho_{k}-\frac{(1-\lambda_{k})\alpha_{k}^{2}}
{\lambda_{k}}+\frac{\alpha_k (1-\lambda_k)}{\lambda_k \tau_k}$,
$\omega_{k}=\frac{(1-\lambda_k)(\alpha_k\tau_k -1)}{\lambda_k}$, and $M_k = \delta_{k}\lambda_{k} +  2(1-\lambda_k)\|  z^{k+1}-u^k \| $. Notice that $\tau_{k}=\frac{1}{\alpha_{k}+\delta\lambda_{k}}$, then $\delta = \frac{1-\tau_k \alpha_k}{\tau_k \lambda_k}$. We have $0<\varsigma_{k}\leq\alpha(1+\alpha) + \alpha \delta$ and $\omega_k > 0$. From the condition (c) $\{z^k\}$ is bounded, then there exists a constant $M>0$ such that $M_k \leq M$, for any $k$.

In the following, we follow the same technique used in \cite{Bot2015AMC}. Let $\phi^{k}=\psi^{k}-\alpha_{k}\psi^{k-1}+\varsigma_{k}\|z^{k}-z^{k-1}\|^{2}$. By condition (a), we have
\begin{align}\label{eq32-6}
\phi^{k+1}-\phi^{k}&\leq\psi^{k+1}-(1+\alpha_{k})\psi^{k}+\alpha_{k}\psi^{k-1}+\varsigma_{k+1}\|z^{k+1}-z^{k}\|^{2}-\varsigma_{k}\|z^{k}-z^{k-1}\|^{2}\nonumber\\
&\leq M\|e^{k}\| +\varsigma_{k}\|z^{k-1}-z^{k}\|^{2}+\omega_{k}\|z^{k+1}-z^{k}\|^{2}+\varsigma_{k+1}\|z^{k+1}-z^{k}\|^{2}-\varsigma_{k}\|z^{k}-z^{k-1}\|^{2}\nonumber\\
& = M\|e^{k}\|+(\omega_{k}+\varsigma_{k+1})\|z^{k+1}-z^{k}\|^{2}.
\end{align}

Next, we claim that
\begin{equation}\label{eq32-7}
\omega_{k}+\varsigma_{k+1}\leq-\sigma,  \textrm{ for some } \sigma > 0.
\end{equation}
In fact,
\begin{align}\label{eq32-8}
&\frac{(1-\lambda_{k})(\alpha_{k}\tau_{k}-1)}{\lambda_{k}}+\varsigma_{k+1}\leq -\sigma  \nonumber\\
&\Longleftrightarrow (1-\lambda_{k})(\alpha_{k}\tau_{k}-1)+\lambda_{k}(\varsigma_{k+1}+\sigma)\leq 0  \nonumber\\
&\Longleftrightarrow
(1-\lambda_k)\frac{-\delta \lambda_k}{\alpha_k + \delta \lambda_k} + \lambda_{k}(\varsigma_{k+1}+\sigma)\leq 0  \nonumber\\
&\Longleftrightarrow (\alpha_{k}+\delta\lambda_{k})(\alpha(1+\alpha) + \alpha \delta +\sigma)+\delta \lambda_{k} \leq \delta,
\end{align}
which is true by taking into account the condition of (a) and (b). Then, we obtain from (\ref{eq32-6}) that
\begin{align}\label{eq32-9}
\phi^{k+1}-\phi^{k} & \leq M\|e^{k}\|-\sigma\|z^{k+1}-z^{k}\|^{2}.
\end{align}
Therefore, we get
\begin{equation}\label{eq32-10}
\phi^{k+1} \leq \phi^{1} + M \sum_{i=1}^{k}\|e^{i}\|.
\end{equation}
It follows from $\sum_{k=1}^{+\infty}\|e^{k}\| < +\infty$ that $\phi^{k}<+\infty$ for any $k$. Let $\overline{\phi} >0$ such that $\phi^{k} \leq \overline{\phi}$.

On the other hand, we obtain from the definition of $\phi^{k}$ that
\begin{equation}\label{eq32-11}
\phi^{k} \geq  \psi^{k}-\alpha_{k}\psi^{k-1} \geq \psi^k - \alpha \psi^{k-1} \geq - \alpha \psi^{k-1}.
\end{equation}
Therefore, we get
\begin{equation}\label{eq32-12}
\psi^k \leq \alpha \psi^{k-1} + \phi^k \leq \alpha^k \psi^{0} + \sum_{i=0}^{k-1}\alpha^{i} \phi^{k-i} \leq \alpha^k \psi^{0} + \frac{1}{1-\alpha} \overline{\phi}.
\end{equation}
Further, we obtain from (\ref{eq32-9}), (\ref{eq32-11}) and (\ref{eq32-12}) that
\begin{align}
\sigma \sum_{k=1}^{n}\| z^{k+1} - z^k \|^2 & \leq \phi^{1} - \phi^{n+1} + M \sum_{k=1}^{n}\|e^k\| \nonumber \\
& \leq \phi^{1} + \alpha \psi^{n} + M \sum_{k=1}^{n}\|e^k\| \nonumber \\
& \leq \phi^{1} + \alpha^{n+1} \psi^{0}  + \frac{\alpha}{1-\alpha}\overline{\phi} + M \sum_{k=1}^{n}\|e^k\| \nonumber \\
\end{align}
which means that
\begin{equation}
\sum_{k=1}^{+\infty}\| z^{k+1} - z^k \|^2 < +\infty.
\end{equation}
Consequently, we have $\lim_{k\rightarrow +\infty}\| z^{k+1} - z^k \| = 0$.

(ii) By (i) and (\ref{eq3-7}), from Lemma \ref{our-lemma2}, we get that $\lim_{k\rightarrow +\infty}\|z^{k}-z^{*}\|$ exists. Besides, by (i) and
 $\mu^{k}=z_{k}+\alpha_{k}(z^{k}-z^{k-1})$, then
\begin{align}
\|\mu^{k}-z^{k+1}\|&\leq\|z^{k}-z^{k+1}\|+\alpha_{k}\|z^{k}-z^{k-1}\|\nonumber\\
&\leq\|z^{k}-z^{k+1}\|+\alpha\|z^{k}-z^{k-1}\|
\rightarrow 0,  \textrm{ as } k \rightarrow +\infty.
\end{align}

(iii) By (ii) and condition (b), we have
\begin{align}
\|T\mu^{k}-\mu^{k}\| & = \left\|\frac{z^{k+1}-\mu^{k}}{\lambda_{k}}-e_{k} \right\| \nonumber \\
& \leq \frac{\|z^{k+1}-\mu^{k}\|}{\lambda}+\|e_{k}\|  \rightarrow 0, \textrm{ as } k \rightarrow +\infty,
\end{align}
and
\begin{equation}
\|\mu^{k}-z^{k}\|\leq \|\mu^{k}-z^{k+1}\|+\|z^{k+1}-z^{k}\| \rightarrow 0   \textrm{ as }  k \rightarrow +\infty.
\end{equation}

Since the reason of proving that $\{z^{k}\}$ converges weakly to a fixed point of $T$ is the same as Theorem \ref{main-theorem 1}, so we omit it here.
\end{proof}

\begin{remark}
The proposed  inexact  inertial KM algorithm (\ref{two-step algorithm1}) can also be viewed as a special case of the inexact KM algorithm (\ref{inexact-KM}). In fact, the iteration scheme (\ref{two-step algorithm1}) can be rewritten as follows
\begin{equation}
z^{k+1} = z^k + \lambda_k (Tz^k - z^k + \overline{e}^k),
\end{equation}
where $\overline{e}^k = \frac{\alpha_k}{\lambda_k}(z^k - z^{k-1}) + T\mu^k - Tz^k - \alpha_k (z^k - z^{k-1}) + e^k$. Then, we have
\begin{align}
\lambda_k \|\overline{e}^k\| & = \|\alpha_k(z^k - z^{k-1}) + \lambda_k( T\mu^k - Tz^k) - \lambda_k \alpha_k (z^k - z^{k-1}) + \lambda_k e^k  \| \nonumber \\
& \leq \alpha_k \|z^k - z^{k-1}\| + 2\lambda_k \alpha_k \|z^k - z^{k-1}\| + \lambda_k \|e^k\|.
\end{align}
If we assume that $\sum_{k=0}^{+\infty}\alpha_k \|z^k - z^{k-1}\| < +\infty $ and $\sum_{k=0}^{+\infty}\lambda_k\|e^k\| < +\infty$, then the convergence of $\{z^k\}$ can be directly obtained from the inexact KM algorithm (\ref{inexact-KM}). It is obvious that the convergence conditions used in Theorem \ref{main-theorem 1} and Theorem \ref{main-theorem 2} are weaker than these conditions mentioned before. In addition,
Theorem \ref{main-theorem 2} requires a stronger condition on the error term than Theorem \ref{main-theorem 1}.

\end{remark}

\begin{remark}

(i)\, Let $\alpha_k =0$, then the proposed inexact inertial KM algorithm (\ref{two-step algorithm1}) reduces to the classical inexact KM algorithm (\ref{inexact-KM}). Since the iterative sequence generated by the inexact KM algorithm is quasi-Fej\'{e}r monotone, which is bounded. Therefore, the condition of $\{z^k\}$ is bounded in Theorem \ref{main-theorem 1} and Theorem \ref{main-theorem 2} are naturally satisfied.

(ii)\, Let $e_k =0$, then (\ref{two-step algorithm1}) recovers the inertial KM algorithm (\ref{inertial KM}). We can also remove the condition of $\{z^k\}$ is bounded, which follows the same proof of \cite{Mainge2008JCAM, Bot2015AMC}.

\end{remark}

In the following, we prove the convergence rate of the proposed inexact inertial KM algorithm (\ref{two-step algorithm1}). The following theorem is obtained in the spirit of Shehu \cite{Shehu2018NFAO}.

\begin{theorem}\label{main-theorem-rate}
Let $H$ be a real Hilbert space. Let $T:H\rightarrow H$ is a nonexpansive operator such that $Fix(T)\neq \varnothing$. For any given $z^{0}, z^{-1}\in H$,
let the iterative sequences $\{z^{k}\}, \{\mu^{k}\}$ are defined by (\ref{two-step algorithm1}). Assume that the parameters $\{\lambda_{k}\}$ and $\{\alpha_{k}\}$  satisfy one of the conditions of Theorem \ref{main-theorem 1} and Theorem \ref{main-theorem 2}, respectively. Then, for any $z^{*}\in Fix(T)$, we have
\begin{equation}
\min_{1\leq i \leq k} \| T\mu^{i} - \mu^{i} \|^2 \leq \frac{1}{k \lambda (1-\theta)} (\|z^{1} - z^{*}\| + \Delta ),
\end{equation}
where $\Delta = \alpha \sum _{k=1}^{\infty}[\|z^{k}-z^{*}\|^{2}- \|z^{k-1}-z^{*}\|^{2}]_{+} + \sum_{k=1}^{+\infty}\delta_k \lambda_k \|e^k\| + \sum_{k=1}^{+\infty}\rho_k \| z^k - z^{k-1} \|^2 < +\infty$, and $\theta = \lambda^{'}$ or $\frac{\delta-\alpha[\alpha (1+\alpha)+\alpha \delta+\sigma]}{\delta[1+\alpha (1+\alpha)+\alpha \delta+\sigma]}$, respectively.
\end{theorem}

\begin{proof}
By (\ref{eq3-8}), we have
\begin{align}\label{}
 \lambda_{k}(1-\lambda_{k})\|T\mu^{k}-\mu^{k}\|^{2}&\leq \|z^{k}-z^{*}\|^{2}-\|z^{k+1}-z^{*}\|^{2}+\alpha_{k}(\|z^{k}-z^{*}\|^{2}-\|z^{k-1}-z^{*}\|^{2})\nonumber\\
 &+\delta_{k}\lambda_{k}\|e^{k}\|+\rho_{k}\|z^{k}-z^{k-1}\|^{2} \nonumber \\
 & \leq \|z^{k}-z^{*}\|^{2}-\|z^{k+1}-z^{*}\|^{2}+\alpha_{k}[\|z^{k}-z^{*}\|^{2}- \|z^{k-1}-z^{*}\|^{2}]_{+} \nonumber \\
 & + \delta_{k}\lambda_{k}\|e^{k}\|+\rho_{k}\|z^{k}-z^{k-1}\|^{2}.
\end{align}
Therefore, we get
\begin{align}\label{}
 \lambda(1-\theta)\sum_{i=1}^{k}\|T\mu^{i}-\mu^{i}\|^{2}&\leq \|z^{1}-z^{*}\|^{2}-\|z^{k+1}-z^{*}\|^{2}+\alpha \sum_{i=1}^{k}(\|z^{i}-z^{*}\|^{2}-\|z^{i-1}-z^{*}\|^{2})\nonumber\\
 &+\sum_{i=1}^{k}\delta_{i}\lambda_{i}\|e^{i}\|+ \sum_{i=1}^{k}\rho_{i}\|z^{i}-z^{i-1}\|^{2} \nonumber \\
 & \leq  \| z^1 - z^{*} \|^2 + \Delta,
\end{align}
where $\Delta = \alpha \sum _{k=1}^{\infty}[\|z^{k}-z^{*}\|^{2}- \|z^{k-1}-z^{*}\|^{2}]_{+} + \sum_{k=1}^{+\infty}\delta_k \lambda_k \|e^k\| + \sum_{k=1}^{+\infty}\rho_k \| z^k - z^{k-1} \|^2 < +\infty$. Then, we arrive at
$$
\min_{1\leq i \leq k} \| T\mu^{i} - \mu^{i} \|^2 \leq \frac{1}{k \lambda (1-\theta)} (\|z^{1} - z^{*}\| + \Delta ).
$$
This completes the proof.

\end{proof}

\begin{remark}
In Theorem \ref{main-theorem-rate}, we provide the convergence rate of the inexact inertial KM algorithm (\ref{two-step algorithm1}) in the sense of $\min_{1\leq i \leq k} \| T\mu^{i} - \mu^{i} \|^2$. While the convergence rate for the KM algorithm (\ref{KM-iteration}) and the inexact KM algorithm (\ref{inexact-KM}) is usually measured by $\|T\mu^k - \mu^k\|$. See, for instance \cite{Cominetti2014,Liang2016F,Matsushita2017}. Therefore, the convergence rate result for the inexact inertial KM algorithm (\ref{two-step algorithm1}) is weaker than the KM algorithm (\ref{KM-iteration}) and the inexact KM algorithm (\ref{inexact-KM}).

\end{remark}

\section{Applications}

In this section, we study several applications of the proposed algorithm (\ref{two-step algorithm1}) for solving monotone inclusion problems and the corresponding convex minimization problems, respectively.

First, we consider the simplest monotone inclusion problem:
\begin{equation}\label{monotone-inclusion-1}
\textrm{ find }\, x\in H, \textrm{ such that }\, 0\in Ax,
\end{equation}
where $H$ is a real Hilbert space, and $A:H \rightarrow 2^{H}$ is maximally monotone operator.

The following convex minimization problem is closely related to the monotone inclusion problem (\ref{monotone-inclusion-1}).
\begin{equation}\label{convex-minimization-1}
\min_{x\in H}\ f(x),
\end{equation}
where $f:H \rightarrow (-\infty, +\infty]$ is a proper closed lower semi-continuous convex function. The most well-known algorithm for solving the monotone inclusion problem (\ref{monotone-inclusion-1}) and the convex minimization problem (\ref{convex-minimization-1}) is the proximal point algorithm (PPA) \cite{Eckstein1992}.  Next we will apply the proposed inexact inertial KM algorithm (\ref{two-step algorithm1}) to solve the problem (\ref{monotone-inclusion-1}) and (\ref{convex-minimization-1}), respectively.

\begin{theorem}\label{main-theorem 3}
Let $H$ be a real Hilbert space. Let $A:H \rightarrow 2^{H}$ is maximally monotone operator.  Suppose that $\Omega :=zer A \neq \varnothing$. Let the iterative sequences $\{z^{k}\}, \{\mu^{k}\}$ are generated by following scheme:
\begin{equation}\label{ppa-algorithm1}
\left\{
\begin{aligned}
\mu^{k}&=z^{k}+\alpha_{k}(z^{k}-z^{k-1}), \\
z^{k+1}&=\mu^{k}+\lambda _{k}(J_{\rho A}\mu^{k}+e^{k}-\mu^{k}),
\end{aligned}
\right.
\end{equation}
where $\rho >0$. Assume that the parameters $\lambda_{k}$ and $\alpha_{k}$ satisfy $ 0\leq \alpha_{k} \leq \alpha < 1 $, $ 0\leq \lambda \leq \lambda_{k}\leq \lambda^{'} < 2 $ and the  (b), (c) of conditions (I).
Then the following hold :

\noindent \emph{(i)} $\lim_{k\rightarrow +\infty}\|z^{k}-z^{*}\|$ exists, for any $z^{*}\in \Omega$.

\noindent \emph{(ii)} $\lim_{k\rightarrow +\infty}\|J_{\rho A}\mu^{k}-\mu^{k}\|=0$.

\noindent \emph{(iii)} $\{z^{k}\}$ converges weakly to a zero of $A$.

\noindent \emph{(iv)} $
\min_{1\leq i \leq k} \| T\mu^{i} - \mu^{i} \|^2 \leq \frac{1}{k \lambda (1-\lambda')} (\|z^{1} - z^{*}\| + \Delta )
$, where $\Delta$ is the same as Theorem \ref{main-theorem-rate}.
\end{theorem}

\begin{proof}
Since $J_{\rho A}$ is firmly monexpansive (is also $1/2$-averaged) and $Fix(J_{\rho A})= zer A$. Then there exists an nonexpansive operator $N$ such that $J_{\rho A}=\frac {1}{2}I +\frac {1}{2} N$ and $Fix(J_{\rho A})=Fix(N)$. Therefore the iteration scheme $\{z^{k+1}\}$ of  (\ref{ppa-algorithm1}) can be rewritten as
\begin{equation}
 z^{k+1}=\mu^{k}+\frac{1}{2}\lambda _{k}(N\mu^{k}+2e^{k}-\mu^{k}).
\end{equation}
 Then we can get the conclusions (i), (ii) and (iii) from Theorem \ref{main-theorem 1} and (iv) from Theorem \ref{main-theorem-rate} immediately.
\end{proof}

Similar to Theorem \ref{main-theorem 3}, we obtain the following convergence theorem from Theorem \ref{main-theorem 2}. Since the proof is the same as Theorem \ref{main-theorem 3}, so we omit it here.

\begin{theorem}\label{main-theorem 4}

Let $H$ be a real Hilbert space. Let $A:H \rightarrow 2^{H}$ is maximally monotone operator such that $\Omega :=zer A \neq \varnothing$. Let the iterative sequences $\{z^{k}\}, \{\mu^{k}\}$ are generated by (\ref{ppa-algorithm1}).
Assume that the parameters $\lambda_{k}$ and $\alpha_{k}$ satisfy:  let $\lambda, \sigma, \delta > 0 $ such that
\begin{equation}
\delta > \frac{\alpha[\alpha (1+\alpha)+\sigma]}{1-\alpha^{2}} \textrm{ and } 0<\lambda \leq \lambda_{k}\leq \overline{\lambda} :=2\frac{\delta-\alpha[\alpha (1+\alpha)+\alpha \delta+\sigma]}{\delta[1+\alpha (1+\alpha)+\alpha \delta+\sigma]};
\end{equation}
and the (a),(c) (d) from conditions (II).
Then the following hold:

\noindent \emph{(i)} $\sum _{k=0}^{\infty}\|z^{k+1}-z^{k}\|^{2} < +\infty$. Moreover, $\lim_{k\rightarrow +\infty}\|z^{k+1}-z^{k}\|=0$.

\noindent \emph{(ii)} $\lim_{k\rightarrow +\infty}\|z^{k}-z^{*}\|$ exists, for any $z^{*}\in \Omega$.

\noindent \emph{(iii)} $\lim_{k\rightarrow +\infty}\|J_{\rho A}\mu^{k}-\mu^{k}\|=0$ and $\{z^{k}\}$ converges weakly to a point in $zer A$.

\noindent \emph{(iv)} $
\min_{1\leq i \leq k} \| T\mu^{i} - \mu^{i} \|^2 \leq \frac{1}{k \lambda (1-\overline{\lambda})} (\|z^{1} - z^{*}\| + \Delta )
$, where $\Delta$ is the same as Theorem \ref{main-theorem-rate}.
\end{theorem}

\begin{corollary}\label{ppa corollary1}
Let $H$ be a real Hilbert space. Let $f:H\rightarrow (-\infty,+\infty]$ is a proper, closed lower semi-continuous, convex function. Suppose that $\Omega :=Argmin f\neq \varnothing$. Let the iterative sequences $\{z^{k}\}, \{\mu^{k}\}$ are generated by following algorithm:
\begin{equation}\label{ppa-1}
\left\{
\begin{aligned}
\mu^{k}&=z^{k}+\alpha_{k}(z^{k}-z^{k-1}), \\
z^{k+1}&=\mu^{k}+\lambda _{k}(prox_{\rho f}\mu^{k}+e^{k}-\mu^{k}).
\end{aligned}
\right.
\end{equation}
Assume that the parameters $\lambda_{k}, \alpha_{k}$ satisfy the conditions from Theorem \ref{main-theorem 3} or Theorem \ref{main-theorem 4}.
Then the following hold:

\noindent \emph{(i)} $\lim_{k\rightarrow +\infty}\|z^{k}-z^{*}\|$ exists, for any $z^{*}\in \Omega$.

\noindent \emph{(ii)} $\{z^{k}\}$ converges weakly to a solution the convex minimization problem (\ref{convex-minimization-1}).

\noindent \emph{(iii)} $
\min_{1\leq i \leq k} \| T\mu^{i} - \mu^{i} \|^2 \leq \frac{1}{k \lambda (1-\theta)} (\|z^{1} - z^{*}\| + \Delta )
$, where $\theta = \lambda'$ or $\theta = \overline{\lambda}$, and $\Delta$ is the same as Theorem \ref{main-theorem-rate}.
\end{corollary}

\begin{proof}
Since the subdifferential of a proper, convex and lower semi-continuous function is maximally monotone operator, then $\partial f$ is maximally monotone.
Because of $prox_{\rho f}=J_{\rho \partial f}$, Then we can get the conclusions (i)-(iii) from Theorem \ref{main-theorem 3} and Theorem \ref{main-theorem 4}, respectively.
\end{proof}

\begin{remark}
The newly obtained inexact inertial proximal point algorithm includes  proximal point algorithms as a special case.

(1) Let $\alpha_k = 0$, then (\ref{ppa-algorithm1}) reduces to the classical proximal point algorithm with error studied in \cite{RockafellarSIAM1976,Eckstein1992};

(2) Let $e^k =0$, then (\ref{ppa-algorithm1}) recovers the relaxed inertial proximal point algorithm in \cite{attouch:hal-01708905}.

\end{remark}

Second, we consider solving the following monotone inclusion problem:
\begin{equation}\label{monotone-inclusion-2}
\textrm{ find }\, x\in H, \textrm{ such that }\, 0\in Ax+Bx,
\end{equation}
where $A:H\rightarrow 2^H$ is maximally monotone operator and $B:H\rightarrow H$ is a $\beta$-inverse strongly monotone operator, for some $\beta >0$.
The corresponding convex optimization problem to the monotone inclusion problem (\ref{monotone-inclusion-2}) is
\begin{equation}\label{convex-minimization-2}
\min_{x\in H} f(x)+g(x),
\end{equation}
where $f:H\rightarrow R$ is convex differentiable with a $\frac{1}{\beta}$-Lipschitz continuous gradient, and $g:H\rightarrow (-\infty,+\infty]$ is a proper closed lower semi-continuous convex function. We always assume the solution of (\ref{monotone-inclusion-2}) and (\ref{convex-minimization-2}) are not empty. In the following, we apply the proposed inexact inertial KM algorithm to solve (\ref{monotone-inclusion-2}) and (\ref{convex-minimization-2}).

\begin{theorem}\label{thorem FB}
Let $H$ be a real Hilbert space. Let $A:H\rightarrow 2^H$ is maximally monotone operator and $B:H\rightarrow H$ is a $\beta$-inverse strongly monotone operator.  Suppose that $\Omega :=zer (A+B) \neq \varnothing$. Let the iterative sequences $\{z^{k}\}$ and $\{\mu^{k}\}$ are generated by following scheme:
\begin{equation}\label{FB-algorithm1}
\left\{
\begin{aligned}
\mu^{k}&=z^{k}+\alpha_{k}(z^{k}-z^{k-1}), \\
z^{k+1}&=\mu^{k}+\lambda_{k}(J_{\rho A}(\mu^k-\rho (B\mu^{k} + e^{1,k} )+e^{2,k}-\mu^{k}),
\end{aligned}
\right.
\end{equation}
where $\rho \in (0,2\beta)$.
Assume that the parameters $\lambda_{k}$ and $\alpha_{k}$ satisfy $ 0\leq \alpha_{k} \leq \alpha < 1 $, $ 0\leq \lambda \leq \lambda_{k}\leq  \lambda^{'} < \frac{4\beta-\rho}{2\beta} $ and the  (b), (c) of conditions (I), $\sum_{k=0}^{+\infty}\lambda_k \|e^{1,k}\|<+\infty$ and $\sum_{k=0}^{+\infty}\lambda_k \|e^{2,k}\|<+\infty$.
Then the following hold :

\noindent \emph{(i)} $\lim_{k\rightarrow +\infty}\|z^{k}-z^{*}\|$ exists, for any $z^{*}\in \Omega$.

\noindent \emph{(ii)}$\lim_{k\rightarrow +\infty}\|J_{\rho A}(I-\rho B)\mu^{k}-\mu^{k}\|=0$.

\noindent \emph{(iii)} $\{z^{k}\}$ converges weakly to a point in $zer(A+B)$.

\noindent \emph{(iv)} $
\min_{1\leq i \leq k} \| T\mu^{i} - \mu^{i} \|^2 \leq \frac{1}{k \lambda (1-\lambda')} (\|z^{1} - z^{*}\| + \Delta )
$, where $\Delta$ is the same as Theorem \ref{main-theorem-rate}.

\end{theorem}

\begin{proof}
 According to Proposition 26.1 of \cite{bauschkebook2017} that $J_{\rho A}(I-\rho B)$ is $\tau :=\frac{2\beta}{4\beta-\rho}$-averaged and $Fix(J_{\rho A}(I-\rho B)) = zer (A+B)$.  Then there exists an nonexpansive operator $N$ such that $J_{\rho A}(I-\rho B) = (1-\tau)I + \tau N$ and $Fix(J_{\rho A}(I-\rho B)) = Fix(N)$. Therefore, we have
\begin{equation}
z^{k+1}=\mu^{k}+\tau\lambda _{k}(N\mu^{k}+\frac {1}{\tau}(J_{\rho A}(\mu^k-\rho (B\mu^{k} + e^{1,k}))+ e^{2,k}-J_{\rho A}(\mu^{k}-\rho B\mu^{k}))-\mu^{k})
\end{equation}
Notice that $J_{\rho A}$ is nonexpansive, we obtain that
\begin{align}
& \quad \|\frac {1}{\tau}(J_{\rho A}(\mu^k-\rho (B\mu^{k} + e^{1,k}))+ e^{2,k}-J_{\rho A}(\mu^{k}-\rho B\mu^{k}))\| \nonumber \\
& \leq
 \frac {1}{\tau}\|J_{\rho A}(\mu^k-\rho (B\mu^{k} + e^{1,k}))-J_{\rho A}(\mu^{k}-\rho B\mu^{k})\|+ \frac {1}{\tau}\|e^{2,k}\| \nonumber \\
 & \leq \frac {1}{\tau}\|\mu^k-\rho (B\mu^{k} + e^{1,k})-(I-\rho B)\mu^{k}\|+ \frac {1}{\tau}\|e^{2,k}\| \nonumber \\
 & \leq \frac {\rho}{\tau}\| e^{1,k}\|+ \frac {1}{\tau}\|e^{2,k}\|
\end{align}
It is easy to check that the conditions of Theorem \ref{main-theorem 1} are satisfied, then  we can get the conclusions (i), (ii) and (iii) from Theorem \ref{main-theorem 1} and (iv) from Theorem \ref{main-theorem-rate}.
\end{proof}

Similar to Theorem \ref{thorem FB}, we can obtain the following convergence result from Theorem \ref{main-theorem 2}.

\begin{theorem}\label{thorem FB2}
Let $H$ be a real Hilbert space. Let $A:H\rightarrow 2^H$ is maximally monotone operator and $B:H\rightarrow H$ is a $\beta$-inverse strongly monotone operator.  Suppose that $\Omega :=zer (A+B) \neq \varnothing$. Let the iterative sequences $\{z^{k}\}, \{\mu^{k}\}$ are generated by (\ref{FB-algorithm1}).
Assume that the parameters $\lambda_{k}$ and $\alpha_{k}$ satisfy:  Let $\lambda, \sigma, \delta > 0 $ such that
\begin{equation}
\delta > \frac{\alpha[\alpha (1+\alpha)+\sigma]}{1-\alpha^{2}} \textrm{ and } 0< \lambda \leq \lambda_{k}\leq \overline{\lambda} :=\frac {4\beta-\rho}{2\beta}\frac{\delta-\alpha[\alpha (1+\alpha)+\alpha \delta+\sigma]}{\delta[1+\alpha (1+\alpha)+\alpha \delta+\sigma]};
\end{equation}
and the (a), (c), (d) from conditions (II), $\sum_{k=0}^{+\infty}\|e^{1,k}\|<+\infty$ and $\sum_{k=0}^{+\infty} \|e^{2,k}\|<+\infty$.
Then the following hold:

\noindent \emph{(i)}  $\sum _{k=0}^{\infty}\|z^{k+1}-z^{k}\|^{2} < +\infty$ , moreover, $\lim_{k\rightarrow +\infty}\|z^{k+1}-z^{k}\|=0$.

\noindent \emph{(ii)} $\lim_{k\rightarrow +\infty}\|z^{k}-z^{*}\|$ exists, for any $z^{*}\in \Omega$.

\noindent \emph{(iii)} $\lim_{k\rightarrow +\infty}\|J_{\rho A}(I-\rho B)\mu^{k}-\mu^{k}\|=0$ and $\{z^{k}\}$ converges weakly to a point in $zer(A+B)$.

\noindent \emph{(iv)} $
\min_{1\leq i \leq k} \| T\mu^{i} - \mu^{i} \|^2 \leq \frac{1}{k \lambda (1-\overline{\lambda})} (\|z^{1} - z^{*}\| + \Delta )
$, where  $\Delta$ is the same as Theorem \ref{main-theorem-rate}.
\end{theorem}

As applications of Theorem \ref{thorem FB} and Theorem \ref{thorem FB2}, we obtain the following result for solving the convex minimization problem (\ref{convex-minimization-2}).

\begin{corollary}\label{FB  corollary1}

Let $H$ be a real Hilbert space. Let $f:H\rightarrow R$ is convex differentiable with a $\frac{1}{\beta}$-Lipschitz continuous gradient, and $g:H\rightarrow (-\infty,+\infty]$ is a proper closed lower semi-continuous convex function. Suppose that $\Omega :=Argmin (f+g)\neq \varnothing$. Let the iterative sequences $\{z^{k}\}, \{\mu^{k}\}$ are generated by following algorithm:
\begin{equation}\label{FB-1}
\left\{
\begin{aligned}
\mu^{k}&=z^{k}+\alpha_{k}(z^{k}-z^{k-1}), \\
z^{k+1}&=\mu^{k}+\lambda _{k}(prox_{\rho g}(\mu^k - \rho (\nabla f( \mu^{k}) + e^{1,k} )+e^{2,k}-\mu^{k}),
\end{aligned}
\right.
\end{equation}
where $\rho \in (0,2\beta)$.
Assume that the parameters $\lambda_{k}, \alpha_{k}$ and the error sequences $\{e^{1,k}\}$ and $\{e^{2,k}\}$ satisfy the conditions of Theorem \ref{thorem FB} or Theorem \ref{thorem FB2}, respectively.
Then the following hold:

\noindent \emph{(i)} $\lim_{k\rightarrow +\infty}\|z^{k}-z^{*}\|$ exists, for any $z^{*}\in \Omega$.

\noindent \emph{(ii)} $\{z^{k}\}$ converges weakly to a point in $\Omega$.

\noindent \emph{(iii)} $
\min_{1\leq i \leq k} \| T\mu^{i} - \mu^{i} \|^2 \leq \frac{1}{k \lambda (1-\theta)} (\|z^{1} - z^{*}\| + \Delta )
$, where $\theta = \lambda'$ or $\theta = \overline{\lambda}$, $\Delta$ is the same as Theorem \ref{main-theorem-rate}.
\end{corollary}

\begin{proof}
Since the subdifferential of a proper, convex and lower semi-continuous function is maximally monotone operator and $\nabla f$ is $\beta$-inverse strongly monotone operator. Let $A = \partial g$ and $B = \nabla f$, Then we can get the conclusions (i), (ii) and (iii) from Theorem \ref{thorem FB} and Theorem \ref{thorem FB2}, respectively.
\end{proof}

\begin{remark}
We point out the relationship between the proposed inexact inertial forward-backward splitting algorithms (\ref{FB-algorithm1}) and (\ref{FB-1}) with existing forward-backward splitting algorithms.

(1) Let $\alpha_k =0$, (\ref{FB-algorithm1}) becomes the traditional forward-backward splitting algorithm with errors \cite{Combettes2004Optimization,combettes2005}.

(2) Let $e^{1,k}=e^{2,k} = 0$, (\ref{FB-algorithm1}) recovers the relaxed inertial forward-backward splitting algorithm \cite{Attouch2019AMO}.

\end{remark}

\begin{remark}
Based on Theorem \ref{main-theorem 1}, Theorem \ref{main-theorem 2}  and Theorem \ref{main-theorem-rate}, we can also extend them to other operator splitting algorithms, such as Douglas-Rachford splitting algorithm \cite{combettes2007}, Generalized forward-backward splitting algorithm \cite{Raguet-SIAM-2013}, and Davis-Yin's three-operator splitting algorithm \cite{davis2015}, etc. To save the space of this paper, we don't present these results here. However, we will discuss them in a more general setting and together with an application to convex optimization problems arising in signal and image processing.
\end{remark}

\section{Conclusions}

To incorporate error in the iterative sequence, we proposed an inexact inertial
Krasnoselskii-Mann algorithm (\ref{two-step algorithm1}) for finding fixed points of nonexpansive operators. Compared with the original inertial Krasnoselskii-Mann algorithm (\ref{inertial KM}), the proposed algorithm generated a sequence, which takes into account the presence of perturbations. We proved the convergence of the proposed algorithm and provided a nonasymptotic convergence rate analysis for it.  As applications, we employed the proposed algorithm to solve monotone inclusion problems and obtained several new algorithms including inexact inertial proximal point algorithm (\ref{ppa-algorithm1}) and inexact inertial forward-backward splitting algorithm (\ref{FB-algorithm1}). These algorithms generalized the famous proximal point algorithm and forward-backward splitting algorithm. In the future, we will further report numerical experiment results for solving convex optimization problems to demonstrate the advantage of it.

\section*{Acknowledgement}
This work is supported by the National Natural Science Foundation of China (11771198,11661056,11401293), the China Postdoctoral Science Foundation (2015M571989)
and the Jiangxi Province Postdoctoral Science Foundation (2015KY51).


\end{document}